\documentclass[10pt]{amsart}

\usepackage[T1]{fontenc}
\usepackage{lmodern}
\usepackage{amsmath,amssymb,mathtools}
\usepackage[hidelinks]{hyperref}
\hypersetup{
  pdftitle={Solubility from a disconnected common-divisor graph on p-regular conjugacy-class sizes},
  pdfauthor={Zhi-Lin Zhang},
  pdfsubject={Finite group theory},
  pdfkeywords={finite groups, conjugacy class sizes, p-regular elements, solubility}
}

\newtheorem{theorem}{Theorem}[section]
\newtheorem{proposition}[theorem]{Proposition}
\newtheorem{lemma}[theorem]{Lemma}
\newtheorem{corollary}[theorem]{Corollary}
\numberwithin{equation}{section}

\title[Disconnected graphs on $p$-regular class sizes]
{Solubility from a disconnected common-divisor graph on $p$-regular conjugacy-class sizes}

\author[Zhi-Lin Zhang]{Zhi-Lin Zhang}
\address{Independent Researcher, Taipei, Taiwan}
\email{hsa00000@gmail.com}
\date{}


\begin{document}

\begin{abstract}
Let $G$ be a finite group and let $p$ be a prime. We prove that if the
common-divisor graph on the nontrivial conjugacy-class sizes of $p$-regular
elements of $G$ is disconnected, then $G$ is soluble. This settles the
remaining case identified by Camina and collaborators. Their structure results
provide an abelian normal $p'$-subgroup $M$. We choose a noncentral
$r$-element $x\in M$, with $r\ne p$, such that $[G:C_G(x)]$ is coprime to
$pr$, and show that $C_G(x)/M$ is a $\{p,r\}$-group. Thus $C_G(x)/M$ is a
Hall $\{p,r\}$-subgroup of $G/M$. The separation between the two graph
components then implies that every $p$-regular conjugacy-class size in $G/M$
has all its prime divisors in $\{p,r\}$. A theorem of Dolfi and Lucido,
together with Burnside's $p^a q^b$-theorem, yields the solubility of $G/M$
and hence of $G$.
\end{abstract}

\maketitle

\section{Introduction}
\label{sec:introduction}

All groups in this paper are finite. Let $p$ be a prime. An element is
\emph{$p$-regular}, or a \emph{$p'$-element}, if its order is coprime to $p$.
For $x\in G$, write $x^G$ for its conjugacy class, so that
$|x^G|=[G:C_G(x)]$. The common-divisor graph $\Gamma_p(G)$ has as vertices
the integers $|x^G|>1$ arising from $p$-regular elements $x\in G$, with two
distinct vertices adjacent when they have a common prime divisor. We consider
the question of whether disconnectedness of this graph forces $G$ to be
soluble.

Beltr\'{a}n and Felipe studied the disconnected case under the additional
hypothesis that $G$ is $p$-soluble~\cite{BeltranFelipe2004}. More recently,
Camina, Mar\'{o}ti, Pacifici, Parker, Rekv\'{e}nyi, Saunders, Sotomayor,
Tracey and van Beek obtained structure results without that hypothesis
\cite{CaminaEtAl2026}. Their Corollaries~4.1--4.3 show that a disconnected
$\Gamma_p(G)$ has exactly two complete components and imply solubility in all
but one case: the case in which $p$ divides a vertex in the component
containing the largest $p$-regular class size. Their Remark~4.4 leaves this
case open. We prove that solubility holds there as well.

\begin{theorem}\label{thm:main}
Let $G$ be a finite group and let $p$ be a prime. If $\Gamma_p(G)$ is
disconnected, then $G$ is soluble.
\end{theorem}

The main difficulty is that disconnection gives only a separation between
the prime divisors occurring in the two components. In contrast, the
Dolfi--Lucido theorem used at the end of the proof requires a fixed prime to
be absent from every $p$-regular conjugacy-class size of the group to which
it is applied. Thus the componentwise information supplied by
$\Gamma_p(G)$ does not directly give the uniform restriction needed for that
theorem.

The structure results of~\cite{CaminaEtAl2026} provide the ingredient needed
to overcome this mismatch, but not the desired global restriction inside
$G$ itself. Let $B$ be a $p$-regular class of largest size and let $X_1$ be
the component not containing $|B|$. Those results yield an abelian normal
$p'$-subgroup $M$ containing every $p$-regular element whose class size lies
in $X_1$, together with a restriction on the primes dividing
$M/(M\cap Z(G))$. The decisive change of viewpoint is to pass to the quotient
$G/M$, where the separation of the two components can be converted into a
restriction on all $p$-regular class sizes.

More precisely, we choose a noncentral $r$-element $x\in M$, with $r\ne p$,
such that $|x^G|\in X_1$, $r\mid |B|$, and
$\gcd(|x^G|,pr)=1$. The $r$-power order of $x$ allows us to show that
$C_G(x)/M$ is a $\{p,r\}$-group: if an $s$-element outside $M$
centralized $x$ for some $s\notin\{p,r\}$, multiplication by $x$ would
produce a $p$-regular class size joining the two components. It follows that
$C_G(x)/M$ is a Hall $\{p,r\}$-subgroup of $G/M$. Thus every prime outside
$\{p,r\}$ dividing $|G/M|$ must divide the fixed vertex $|x^G|\in X_1$.
On the other hand, every nonidentity $p$-regular element of $G/M$ lifts to a $p$-regular element of $G$ whose class size lies
in the other component. Quotient class-size divisibility then shows that
every $p$-regular class size in $G/M$ has all its prime divisors in
$\{p,r\}$. The Dolfi--Lucido theorem, followed by Burnside's
$p^a q^b$-theorem, yields the solubility of $G/M$ and hence of $G$.

\section{External results and a solubility criterion}
\label{sec:literature}

For a positive integer $n$, let $\pi(n)$ be its set of prime divisors, and
put $\pi(Y)=\pi(|Y|)$ for a finite group $Y$. For a set of primes $\sigma$,
a $\sigma$-number is a positive integer whose prime divisors lie in
$\sigma$, and a $\sigma$-group is a finite group of $\sigma$-number order.
Write $\sigma'$ for the complementary set of primes. A Hall
$\sigma$-subgroup has $\sigma$-number order and $\sigma'$-number index.

Let $O_{p'}(G)$ denote the largest normal $p'$-subgroup of $G$, and write
$Z_{p'}=Z(G)\cap O_{p'}(G)$. Equivalently, $Z_{p'}$ is the subgroup of
$p$-regular elements of $Z(G)$.

We first record the structural result for disconnected $\Gamma_p(G)$.
It provides the subgroup $M$ that will be factored out and the
prime-divisor restriction used later to choose $r$.

\begin{proposition}\label{prop:disconnected-structure}
Suppose that $G$ is finite, $p$ is a prime, and $\Gamma_p(G)$ is
disconnected. Then it has exactly two connected components, both complete.
Let $B$ be a $p$-regular conjugacy class of maximal size, label the
components $X_1,X_2$ so that $|B|\in X_2$, and define
\begin{equation}\label{eq:definition-M}
 M=\left\langle
 D:\ D\text{ is a $p$-regular conjugacy class and }
       \gcd(|D|,|B|)=1
 \right\rangle.
\end{equation}
Then $M$ is an abelian normal $p'$-subgroup of $G$, and
\begin{equation}\label{eq:structure-M}
 Z_{p'}\le M,\qquad
 \pi(M/Z_{p'})\subseteq\pi(|B|).
\end{equation}
Moreover, $G$ is soluble if $p$ divides no vertex of $\Gamma_p(G)$, or if
$p$ divides a vertex in $X_1$.
\end{proposition}

The assertion about the connected components is
\cite[Corollary~4.1]{CaminaEtAl2026}. Corollary~4.2 of
\cite{CaminaEtAl2026} gives that $M$ is an abelian $p'$-subgroup and that
$Z_{p'}\le M$ with $\pi(M/Z_{p'})\subseteq\pi(|B|)$. Since $M$ is
generated by conjugacy classes, it is normal in $G$. The two solubility
cases follow from \cite[Corollary~4.3]{CaminaEtAl2026} and are stated
explicitly in \cite[Remark~4.4]{CaminaEtAl2026}. These results do not
require $G$ to be $p$-soluble.

Since $Z_{p'}\le M$, we have $Z_{p'}\le M\cap Z(G)$, while the reverse
inclusion follows because $M$ is a $p'$-group. Hence
\begin{equation}\label{eq:central-intersection}
 M\cap Z(G)=Z_{p'}.
\end{equation}

The second external result is a solubility theorem of Dolfi and Lucido.
For a finite group $Y$ and a prime $q$, write
\[
 O^{q'}(Y)=\langle Q:Q\in\operatorname{Syl}_q(Y)\rangle.
\]
This is the least normal subgroup of $Y$ with $q'$-quotient; in particular,
$Y/O^{q'}(Y)$ has order coprime to $q$.

\begin{theorem}[Dolfi--Lucido]\label{thm:Dolfi-Lucido}
Let $Y$ be a finite group, and let $p$ and $q$ be distinct primes. If
$q\nmid |y^Y|$ for every $p$-regular element $y\in Y$, then
$O^{q'}(Y)$ is soluble.
\end{theorem}

This is \cite[Theorem~3]{DolfiLucido2001}. Dolfi and Lucido formulate
the hypothesis as the property $P(p,q)$: every $p'$-element centralizes
some Sylow $q$-subgroup of $Y$. For a $p$-regular element $y$, the condition
$q\nmid |y^Y|=[Y:C_Y(y)]$ is equivalent to $C_Y(y)$ containing a Sylow
$q$-subgroup of $Y$. Thus the formulation above is exactly the hypothesis
$P(p,q)$. In particular, the theorem does not assume that $Y$ is
$p$-soluble.

The following consequence is the form needed later: a restriction of all
$p$-regular class sizes to the two primes $p$ and $r$ already forces
solubility.

\begin{corollary}\label{cor:two-prime-class-sizes}
Let $Y$ be a finite group, and let $p\ne r$ be primes. If
$\pi(|y^Y|)\subseteq\{p,r\}$ for every $p$-regular element $y\in Y$,
then $Y$ is soluble.
\end{corollary}

\begin{proof}
Let $q\in\pi(Y)\setminus\{p,r\}$. By hypothesis,
$q\nmid |y^Y|$ for every $p$-regular element $y\in Y$, and $q\ne p$.
Thus Theorem~\ref{thm:Dolfi-Lucido}, applied with the pair $(p,q)$,
shows that $N_q:=O^{q'}(Y)$ is soluble. By definition of $O^{q'}(Y)$,
the quotient $Y/N_q$ is a $q'$-group.

Suppose that $S$ is a nonabelian composition factor of $Y$. Since
$N_q\lhd Y$, the composition factors of $Y$ are, up to multiplicity,
those of $N_q$ together with those of $Y/N_q$. As $N_q$ is soluble,
all its composition factors are cyclic of prime order. Hence $S$ must
be a composition factor of $Y/N_q$, and therefore $q\nmid |S|$.

This holds for every $q\in\pi(Y)\setminus\{p,r\}$. Since
$\pi(S)\subseteq\pi(Y)$, it follows that $\pi(S)\subseteq\{p,r\}$.
Thus $|S|$ is divisible by at most two primes, so Burnside's
$p^a q^b$-theorem implies that $S$ is soluble, contradicting the fact
that $S$ is nonabelian simple. Hence $Y$ has no nonabelian composition
factors and is soluble.
\end{proof}

The structure results of~\cite{CaminaEtAl2026} used above and
Theorem~\ref{thm:Dolfi-Lucido} both depend on the classification of finite
simple groups. Hence the proof of Theorem~\ref{thm:main} retains this
dependence.

\section{A restriction on a centralizer}
\label{sec:centralizer}

Assume that $\Gamma_p(G)$ is disconnected, and retain the notation
$B,X_1,X_2,M$ from Proposition~\ref{prop:disconnected-structure}.
That proposition already proves the theorem if $p$ divides no vertex, or
if $p$ divides a vertex of $X_1$. We are therefore reduced to the case in
which $p$ divides a vertex of $X_2$.

Since vertices in distinct connected components cannot have a common
prime divisor, we have
\begin{equation}\label{eq:separated-components}
 \gcd(a,b)=1\qquad(a\in X_1,\ b\in X_2).
\end{equation}
Combining this with the definition of $M$, we also obtain
\begin{equation}\label{eq:X1-in-M}
 y\text{ $p$-regular and }|y^G|\in X_1
 \quad\Longrightarrow\quad y\in M.
\end{equation}

Our next goal is to obtain a Hall $\{p,r\}$-subgroup of $G/M$ from a
suitable centralizer. For this, we need a noncentral $r$-element
$x\in M$, so that the argument using $xy$ in Lemma~\ref{lem:Cx-mod-M} can be applied. We also need
$[G:C_G(x)]$ to be coprime to $pr$; once $C_G(x)/M$ is shown to be a
$\{p,r\}$-group, this will make it a Hall $\{p,r\}$-subgroup of $G/M$.
The restriction
$\pi(M/Z_{p'})\subseteq\pi(|B|)$ in~\eqref{eq:structure-M} makes these
requirements compatible.

\begin{lemma}\label{lem:choice-of-x}
There exist a prime $r\ne p$ and a noncentral $r$-element $x\in M$ such that
\[
 r\mid |B|,\qquad |x^G|\in X_1,\qquad
 \gcd(|x^G|,pr)=1.
\]
\end{lemma}

\begin{proof}
Choose a $p$-regular element $d\in G$ with $|d^G|\in X_1$.
By~\eqref{eq:X1-in-M}, $d\in M$. Write $d$ as the product of its commuting
primary components. Since $d$ is noncentral, one of them is noncentral; denote
it by $x$, an $r$-element. Since $x$ is a power of $d$, we have
\[
 C_G(d)\le C_G(x),\qquad 1<|x^G|\mid |d^G|.
\]
Hence $|x^G|$ is a vertex of $X_1$.

Since $x$ is noncentral, \eqref{eq:central-intersection} gives
$x\notin Z_{p'}$. Thus the image of $x$ in $M/Z_{p'}$ is a nonidentity
$r$-element, and so
$r\in\pi(M/Z_{p'})\subseteq\pi(|B|)$ by~\eqref{eq:structure-M}.
Also $r\ne p$, since $M$ is a $p'$-group.

Now $r$ divides the vertex $|B|\in X_2$, and by our standing assumption
$p$ divides some vertex of $X_2$. Since $|x^G|\in X_1$,
\eqref{eq:separated-components} shows that neither $r$ nor $p$ divides
$|x^G|$. Hence $\gcd(|x^G|,pr)=1$.
\end{proof}

Fix $r$ and $x$ as in Lemma~\ref{lem:choice-of-x}, and set
$\sigma=\{p,r\}$. The next lemma shows that no prime outside $\sigma$
divides $|C_G(x)/M|$. Combined with Lemma~\ref{lem:choice-of-x}, this
will make $C_G(x)/M$ a Hall $\sigma$-subgroup of $G/M$. Its proof is
where the prime-power order of $x$ is used.

\begin{lemma}\label{lem:Cx-mod-M}
The quotient $C_G(x)/M$ is a $\sigma$-group.
\end{lemma}

\begin{proof}
Since $M$ is abelian and $x\in M$, we have $M\le C_G(x)$. Suppose that
some prime $s\notin\sigma$ divides $|C_G(x)/M|$. Let
$P\in\operatorname{Syl}_s(C_G(x))$. Since $M\lhd C_G(x)$, the image
$PM/M$ is a Sylow $s$-subgroup of $C_G(x)/M$, and hence $P\not\le M$.
Choose $y\in P\setminus M$. Then $y$ is an $s$-element commuting with $x$.

Since $s\ne p$, the element $y$ is $p$-regular. It is noncentral, since
every central $p$-regular element lies in $Z_{p'}\le M$. Moreover,
\eqref{eq:X1-in-M} shows that $|y^G|\notin X_1$. As $\Gamma_p(G)$ has
exactly two components, it follows that $|y^G|\in X_2$.

Since $s\ne r$, the commuting elements $x$ and $y$ have coprime orders.
Thus each of $x$ and $y$ is a power of $xy$, and hence
\[
 C_G(xy)=C_G(x)\cap C_G(y).
\]
Therefore
\begin{equation}\label{eq:product-divisibilities}
 |x^G|\mid |(xy)^G|,\qquad |y^G|\mid |(xy)^G|.
\end{equation}
The element $xy$ is also $p$-regular, so $|(xy)^G|$ is a vertex of
$\Gamma_p(G)$. The two divisibilities show that this vertex lies in the
same connected component as both $|x^G|\in X_1$ and $|y^G|\in X_2$,
a contradiction.
\end{proof}

\section{Class sizes in the quotient}
\label{sec:quotient}

Put $H=G/M$. Since $M$ is abelian, it remains to prove that $H$ is
soluble. The preceding lemmas give the two restrictions needed in the
quotient. By Lemmas~\ref{lem:choice-of-x} and~\ref{lem:Cx-mod-M},
$C_G(x)/M$ is a Hall $\sigma$-subgroup of $H$ with index $|x^G|\in X_1$.
Thus every prime dividing $|H|$ outside $\sigma$ must divide the fixed
class size $|x^G|$.

On the other hand, \eqref{eq:X1-in-M} shows that a $p$-regular lift of a
nonidentity element of $H$ cannot have its $G$-class size in $X_1$.
The next proposition combines these two facts to show that every
$p$-regular class size in $H$ is a $\sigma$-number. No assumption that
$\Gamma_p(H)$ is disconnected is required.

\begin{proposition}\label{prop:quotient-class-sizes}
Let $H=G/M$ and $T=C_G(x)/M$. Then $T$ is a Hall $\sigma$-subgroup of $H$,
and
\begin{equation}\label{eq:quotient-primes}
 \pi(H)\setminus\sigma=\pi(|x^G|).
\end{equation}
Furthermore, every $p$-regular element $h\in H$ satisfies
$\pi(|h^H|)\subseteq\sigma$.
\end{proposition}

\begin{proof}
By Lemma~\ref{lem:Cx-mod-M}, $T$ is a $\sigma$-group, while
$[H:T]=[G:C_G(x)]=|x^G|$ is a $\sigma'$-number by
Lemma~\ref{lem:choice-of-x}. Hence $T$ is a Hall $\sigma$-subgroup of $H$,
and
\[
 \pi(H)\setminus\sigma=\pi([H:T])=\pi(|x^G|),
\]
which proves~\eqref{eq:quotient-primes}.

Let $1\ne h\in H$ be $p$-regular, and choose $y\in G$ with $h=yM$.
The map $\langle y\rangle\to\langle h\rangle$ has kernel
$\langle y\rangle\cap M$, so
$|y|=|\langle y\rangle\cap M|\,|h|$. Since $M$ is a $p'$-group and
$h$ is $p$-regular, $y$ is also $p$-regular.

Since $h\ne1$, we have $y\notin M$. Thus $y$ is noncentral, because every
central $p$-regular element lies in $Z_{p'}\le M$. By
\eqref{eq:X1-in-M}, we have $|y^G|\notin X_1$, and hence
$|y^G|\in X_2$.

Now $C_G(y)M/M\le C_H(h)$, so
\begin{equation}\label{eq:quotient-divisibility}
 |h^H|
 \mid [H:C_G(y)M/M]
 =[G:C_G(y)M]
 \mid |y^G|.
\end{equation}
If some prime $q\notin\sigma$ divided $|h^H|$, then
$q\in\pi(H)\setminus\sigma$, and~\eqref{eq:quotient-primes} would give
$q\mid |x^G|$. By~\eqref{eq:quotient-divisibility}, we would also have
$q\mid |y^G|$. This contradicts~\eqref{eq:separated-components}, since
$|x^G|\in X_1$ and $|y^G|\in X_2$. Therefore
$\pi(|h^H|)\subseteq\sigma$.

For $h=1$ the conclusion is immediate.
\end{proof}

\begin{proof}[Proof of Theorem~\ref{thm:main}]
By Proposition~\ref{prop:quotient-class-sizes},
$\pi(|h^H|)\subseteq\{p,r\}$ for every $p$-regular element $h\in H=G/M$.
Since $p\ne r$, Corollary~\ref{cor:two-prime-class-sizes} applies and
shows that $H$ is soluble. As $M$ is abelian, $G$ is soluble.

This proves the remaining case; the other cases were already covered by
Proposition~\ref{prop:disconnected-structure}. Hence
Theorem~\ref{thm:main} follows.
\end{proof}

\section*{Declaration of AI use}
OpenAI's ChatGPT was used during the revision of this work to assist with
literature searches, exploratory checking of mathematical arguments, and
drafting and rephrasing portions of the exposition. All mathematical
claims, proofs, and references were independently verified by the author,
who takes full responsibility for the content and the final text.

\end{document}